\documentclass{amsproc}

\usepackage{amsmath,amssymb,setspace}
\usepackage[active]{srcltx}
\usepackage[pagebackref, colorlinks, linkcolor=red, citecolor=blue, urlcolor=blue, hypertexnames=true]{hyperref}
\usepackage[backrefs]{amsrefs}

\allowdisplaybreaks

\newcommand{\Zb}{{\mathbb Z}}
\newcommand{\Nb}{{\mathbb N}}
\newcommand{\Fb}{{\mathbf F}}

\newtheorem{thm}{Theorem}[section]
\newtheorem{cor}[thm]{Corollary}
\newtheorem{lem}[thm]{Lemma}
\newtheorem{question}[thm]{Question}
\newtheorem{definition}[thm]{Definition}
\newtheorem{example}[thm]{Example}
\newtheorem{remark}[thm]{Remark}
\newtheorem{proposition}[thm]{Proposition}

\theoremstyle{definition}

\title{About Got\^{o}'s method showing surjectivity of word maps}

\author{Abdelrhman Elkasapy}
\address{Abdelrhman Elkasapy, MPI-MIS, Inselstra\ss e 22,
04103 Leipzig, Germany, and Mathematics Department, South Valley University, Qena, Egypt}
\email{elkasapy@mis.mpg.de}

\author{Andreas Thom}
\address{Andreas Thom, Univ.\ Leipzig,
PF 100920, 04009 Leipzig , Germany}
\email{andreas.thom@math.uni-leipzig.de}

\begin{document}

\onehalfspace

\begin{abstract}
Let $\Fb$ be the free group on two letters. For $\omega \in \Fb$ we study the associated word map $\omega \colon SU(n) \times SU(n) \to SU(n)$. Extending a method of Got\^{o}, we show that for $\omega$ not in the second derived subgroup $\Fb^{(2)}$ of $\Fb$, there are infinitely many $n \in \Nb$ such that the associated word map $\omega : SU(n) \times SU(n) \rightarrow SU(n)$ is surjective.
\end{abstract}

\maketitle

\section{Introduction}

Let $\Fb$ be the free group on two letters and let $\omega \in \Fb$. The word map $\omega : SU(n) \times SU(n) \rightarrow SU(n)$ is the natural map, which is given by evaluating $\omega$ on the pair of matrices in $SU(n)$. It has been asked by Michael Larsen at the 2008 Spring Central Section Meeting of the AMS in Bloomington whether for every non-trivial $\omega \in \Fb$ and $n \in \Nb$ high enough, the associated word map $\omega \colon SU(n) \times SU(n) \to SU(n)$ is surjective.
The aim of this work is to provide evidence for a positive answer to this question and prove the surjectivity for some classes of word maps. For convenience we restrict our study to the case $SU(n)$, even though our methods extend to other compact Lie groups.

Questions about the size of the image of word maps for general groups $G$ (in place of $SU(n)$) have a long history and led to interesting connections with various fields of mathematics. The first result of general type is a theorem of Amand Borel \cite{borel} asserting that any non-trivial word map is dominant (as a map between affine complex algebraic varieties) if $G$ is a simple algebraic group; in particular its image is Zariski dense.
Despite this general result, the images of word maps can be very small for compact groups. Indeed, the second author showed in \cite{thom} that for fixed $n \in \Nb$ and any neighborhood $U$ of $1_n \in SU(n)$, there exists $\omega \in \Fb \setminus \{e\}$ such that the image of the associated word maps is contained in $U$. This result is already non-trivial for $n=2$ and led to answers to various long-standing questions in non-commutative harmonic analysis \cite{thom}.

Recently, there has been an extensive study of the size of word maps for finite simple groups, see \cites{lar1,lar2,lar3} and the references therein. One of the high points was the proof of the Ore conjecture \cite{MR2654085}, asserting that every element in a non-abelian finite simple group is a commutator.

Let us come back to $G=SU(n)$. We will observe that Larsen's question becomes more complicated if $\omega \in \Fb$ lies deeper in the lower central series. As usual, we define the lower central series by $\Fb^{(0)} := \Fb$ and $\Fb^{(k+1)} := [\Fb^{(k)},\Fb^{(k)}]$. It is easy to see that for $\omega \not \in \Fb^{(1)}$, $\omega \colon SU(n) \times SU(n) \to SU(n)$ is surjective, see for example Lemma \ref{notzero}. Hence, the first non-trivial case is $w(a,b)=[a,b]:=aba^{-1}b^{-1}$, the commutator of the generators of $\Fb$. This case -- unlike for finite simple groups -- was solved by T{\^o}yama already in 1949. He proved that any element in $SU(n)$ can be written as a commutator $[u,v]$ for suitably chosen elements $u,v \in SU(n)$, see ~\cite{MR0035283} for more details. In the same year G\^{o}to put this result in a more general framework, see ~\cite{goto}. We will recall Got\^{o}'s proof and will take Got\^{o}'s method as the basis for the proof of our result which covers all words $\omega \not \in \Fb^{(2)}$. 

For $n \in \Nb$, we denote by ${\rm lpf}(n)$ the least prime factor of $n$.
Our main result is the following:

\begin{thm} \label{mainresult}
Let $\Fb$ be the free group on two generators and $\omega \in \Fb$. If $\omega \not \in \Fb^{(2)}$, then there exists an integer $k \in \Nb$, such that for all $n \in \Nb$ with ${\rm lpf}(n) \geq k$, the word map $\omega\colon SU(n)\times SU(n)\rightarrow SU(n)$ is surjective.
\end{thm}

It is not known to us if the restriction on the integer $n \in \Nb$ in the assumptions of the previous theorem is necessary. We will prove Theorem \ref{mainresult} at the end of Section \ref{mainsec}.
For particular words we can say more. We define the sequence of Engel words by
\begin{eqnarray*} e_0(a,b)=a, \quad e_k(a,b)= [e_{k-1},b], \ \text{for} \ k \geq1.
\end{eqnarray*}
It is easy to see that $e_k \not \in \Fb^{(2)}$ for all $k \in \Nb$ so that the previous theorem applies. However, in this case we can show:
\begin{thm} \label{engeltheorem}
For all $k,n \in \Nb$, the $k$-th Engel word map $e_k \colon SU(n) \times SU(n) \to SU(n)$ is surjective. 
\end{thm}

This result complements results for finite simple groups of Lie type which were obtained by Bandman-Grunewald-Garion in \cite{Engel}. We will prove Theorem \ref{engeltheorem} in the beginning of Section \ref{engelwords}.

\vspace{0.2cm}

The article is structured as follows. Section \ref{prel} contains some preliminaries on the combinatorics of the free group and on Lie theory. We also review some known results concerning word maps and present a streamlined form of Got\^{o}'s proof from 1949. Section \ref{mainsec} contains the proof of the main result. Here, we have to go into some tedious computations with commutators which we think are necessary to control the effect of a base change in the free group on the natural basis of the natural basis of the derived subgroup. In Section \ref{engelwords} contains a study of Engel words and ends with some questions we could not answer so far.

\vspace{0.2cm}

This article contains work which is part of the PhD-project of the first author.

\section{Preliminaries and review of known results} \label{prel}

\subsection{The free group}
\label{introgroups}

Let us fix some notation. Let $\Fb$ be the free group on two generators $a$ and $b$. A word in $a,b$ takes the form $a^{n_1} b^{m_1} \cdots a^{n_k}b^{m_k}$ for $n_i,m_i \in \Zb$. It is well known that $\Fb^{(1)}$ is the free group on the set $S:= \{[a^n,b^m] \mid n,m \in \Zb,nm \neq 0\}$. Indeed, this is a special case of Proposition 4 in Chapter I, \S1.3  of \cite{Serre}.
Note that $[a^n,b^m]^{-1} = [b^m,a^n]$, so that every element in $\Fb^{(1)}$ has a unique expression as a product of commutators $\{[a^n,b^m] \mid n,m \in \Zb, nm \neq 0\} \cup \{[b^n,a^m] \mid n,m \in \Zb, nm \neq 0\}$, such that $[a^n,b^m]$ and $[b^m,a^n]$ do not appear as consecutive letters.

Let $G$ be a group. For a sequence $g_1,\dots,g_n \in G$ we write $\prod_{i=1}^n g_i$ to denote the ordered product $g_1g_2\cdots g_n \in G$. Note that this is non-standard since we do not assume that the $g_i$'s commute.

\subsection{Some Lie theory}

We denote by $SU(n)$ the group of $n \times n$ special unitary matrices and by $1_n \in SU(n)$ the identity matrix.
The subgroup of diagonal matrices in $SU(n)$ is denoted by
$$T := \left\{{\rm diag}(e^{i\theta_1} ,...,e^{i\theta_n})\mid \theta_i\in \mathbb{R},\sum_i \theta_i =0 \right\}.$$ Any element in $SU(n)$ is conjugate to some element in $T$ and $T$ is called maximal torus in $SU(n)$.
The Lie subalgebra of the Lie algebra $\mathfrak{su}(n)$ corresponding to $T$ is the Cartan subalgebra 
$$\mathfrak{h}:=\left \{{\rm diag}(i\theta_1,...,i\theta_n)\mid \theta_i \in \mathbb{R} ,\sum_i \theta_i =0  \right\}.$$
We denote by $\exp \colon \mathfrak{su}(n) \to SU(n)$ the exponential map and note that its restriction to $\mathfrak{h}$ is a homomorphism $\exp \colon \mathfrak{h} \to T$.
We denote by $N(T)$ the normalizer of $T$ in $SU(n)$.
The Weyl group of $SU(n)$ is $$W(T,SU(n))=N(T)/T \simeq S_n$$ and it acts on $\mathfrak{h}$ by permutation of the coordinates. The linearization of this action yields a homomorphism ${\rm Ad} \colon {\mathbb R}[S_n] \to {\rm End}_{\mathbb R}(\mathfrak{h})$, which will play an important role in our study.

A basic property of word maps $\omega \colon SU(n) \times SU(n) \to SU(n)$ is the identity $$\omega(zuz^{*},zvz^*) = z w(u,v)z^*.$$ Hence, in order to show surjectivity, it is enough to show that $T \subset SU(n)$ lies in the image of $\omega$. We will frequently make use of this fact. For more details about compact Lie groups, see ~\cite{Mark}.

\subsection{Review of known results}

Let us now start with an easy observation. As mentioned, it is easy to see that the word maps are surjective for $\omega \not \in \Fb^{(1)}$.

\begin{lem} \label{notzero} Let $\Fb$ be the free group on two generators, $\omega \in \Fb \setminus \Fb^{(1)}$, and $n \in \Nb$. Then, 
the word map $\omega : SU(n)\times SU(n)\rightarrow SU(n)$ is surjective. 
\end{lem}

\begin{proof} We write
$$\omega (a,b)=a^{n_1}b^{m_1}a^{n_2}b^{m_2}...a^{n_k}b^{m_k}$$ with $n_i,m_i\in \mathbb{Z}$ and note that $\sum_i n_i\ne 0$ or $\sum_i m_i\ne 0$. Without loss of generality $\sum_i n_i = k \neq 0$.
If $g\in T$, then $g=h^k$ for some $h\in T$. Indeed, since the exponential map ${\rm exp} \colon \mathfrak{h} \to T$ is a surjective homomorphism, we can take $\bar g \in \mathfrak{h}$ to be some preimage of $g$ and set $h := \exp(\bar g/k)$. Then $\omega(h,1_n)=h^k=g$. We conclude that $\omega$ is a surjective map.
\end{proof}

We will now explain the Got\^{o}'s proof of the main result from ~\cite{goto, MR0035283} -- in the case of $G=SU(n)$.

\begin{thm}[Got\^{o}, T{\^o}yama] \label{lie} Let $n \in \Nb$. The word map $\omega : SU(n)\times SU(n)\rightarrow SU(n)$ with $\omega (a,b)=[a,b]$ is surjective.
\end{thm}
\begin{proof} For the permutation $\sigma=(1,2, \dots, n)\in S_n$, it is easy to see that ${\rm Ad} (\sigma-1)$ is a vector space automorphism of $\mathfrak{h}$. Indeed, it is well-known that the eigenvalues of ${\rm Ad}(\sigma)$ acting on $\mathfrak{h}$ are $\{ \exp(2\pi i l/n) \mid 1 \leq l \leq n-1 \}$.

Now, let $g \in T$ be arbitrary. Since the exponential map $\exp:\mathfrak{h} \rightarrow T$ is surjective, there exists $\bar g\in \mathfrak{h}$ such that $g=\exp(\bar g)$. Since ${\rm Ad} (\sigma-1)$ is automorphism of $\mathfrak{h}$, there is $ \bar h\in \mathfrak{h}$ such that $\bar g= {\rm Ad}(\sigma-1)(h)$. Then, setting $h := \exp(\bar h)$ we get:
$$g= \exp \bar g= \exp({\rm Ad} (\sigma-1)(h)) =\exp( {\rm Ad}(\sigma)h) \exp(-h)= \sigma h \sigma^{-1} h^{-1}=[\sigma,
h].$$ Now, the permutation matrix $\sigma$ might not be in $SU(n)$, however if $\det(\sigma)=-1$, then we just replace $\sigma$ by $\exp(\pi i/n)\sigma \in SU(n)$. This proves the claim.
\end{proof}

\begin{remark}
Note that Got\^{o}'s proof shows the stronger statement that there exists a conjugacy class $C \subset SU(n)$ such that $C^2=SU(n)$. Indeed, for odd $n \in \Nb$ just take $C$ to be the conjugacy class of $\sigma$ and note that $\sigma^{-1} \in C$; similarly for $\exp(\pi i/n)\sigma$ if $n$ is even. In the world of non-abelian finite simple groups, this is known as Thompson's conjecture.
\end{remark}

The idea in Got\^{o}'s proof depends on finding a suitable Laurent polynomial $p(t) \in \mathbb{Z}[t,t^{-1}]$ and a suitable element $\sigma \in W(T)$ in the Weyl group of the maximal torus such that ${\rm Ad}( p(\sigma))$ is a vector space automorphism of the Cartan subalgebra. In the case of the commutator word $\omega=[a,b]$, we take $\sigma=(1,2, \dots ,n)\in S_n$ and $p(t) = t-1$. Our goal is to extend the method to cover more elements in $\Fb$.

\section{The main result} \label{mainsec}

In this section, we want to associate to $\omega \in \Fb^{(1)}$ a polynomial $p_{\omega}$ which can be used in an argument analogous to the one in Got\^{o}'s proof. We define a homomorphism $p_{\omega} \colon \Fb^{(1)} \to \Zb[t,t^{-1}]$
by setting
$$p_{[a^n,b^m]}(t) = m (t^n -1), \quad \forall n,m \in \Zb, nm \neq 0.$$
Note that this is well-defined since $\{ [a^n,b^m] \mid n,m \in \Zb, nm \neq 0\}$ generates $\Fb^{(1)}$ freely, see Section \ref{introgroups}. Since $\Zb[t,t^{-1}]$ abelian, $p_{\omega} = 0$ for all $\omega \in \Fb^{(2)}$.

\begin{lem} \label{large}
Let $\Fb$ be the free group on two generators and let $\omega \in  \Fb^{(1)}$. 
If $p_{\omega}(\exp(2\pi l i/n)) \neq 0$ for $1 \leq l \leq n-1$, then the word map $\omega \colon SU(n)\times SU(n)\rightarrow SU(n)$ is surjective.
\end{lem}
\begin{proof}
We write
$$\omega = [a^{n_1},b^{m_1}]^{\varepsilon_1} \cdots [a^{n_k},b^{m_k}]^{\varepsilon_k}$$
with $n_i,m_i \in \Zb$ and $\varepsilon_i \in \{\pm 1\}$. Then,
$$p_{\omega}(t) = \sum_{i=1}^k \varepsilon_i m_i (t^{n_i}-1).$$ Let $g \in T$ be arbitrary and let $\bar g \in \mathfrak{h}$ be such that $\exp(\bar g)=g$. Let $\sigma = (1,2,\cdots,n) \in W(T)$. By assumption ${\rm Ad}(p_{\omega}(\sigma))$ is invertible in ${\rm End}_{\mathbb R}(\mathfrak{h})$. Let $\bar h \in \mathfrak{h}$ be such that ${\rm Ad}(p_{\omega}(\sigma))(\bar h))= \bar g$ and set $h:= \exp(\bar h)$. We claim that $\omega(\sigma,h)=g$. Indeed,
\begin{eqnarray*}
\omega(\sigma,h) &=& [\sigma^{n_1},h^{m_1}]^{\varepsilon_1} \cdots [\sigma^{n_k},h^{m_k}]^{\varepsilon_k} \\
&=& \prod_{i=1}^k \exp({\rm Ad}(\sigma^{n_i})( \varepsilon_im_i \bar h)) \exp( - \varepsilon_im_i \bar h)\\
&=& \exp({\rm Ad}(p_{\omega}(\sigma)(\bar h))) \\
&=& g.
\end{eqnarray*}
If $n$ is even, then we must replace $\sigma$ by $\exp(\pi i/n) \sigma \in SU(n)$.
This proves the claim.
\end{proof}

\begin{cor} \label{corlem}
Let $\Fb$ be the free group on two generators and let $\omega \in  \Fb^{(1)}$. 
If $p_{\omega} \neq 0$, then there exists an integer $k \in \Nb$, such that for all $n \in \Nb$ with ${\rm lpf}(n) \geq k$, the word map $\omega\colon SU(n)\times SU(n)\rightarrow SU(n)$ is surjective.
\end{cor}
\begin{proof}
Assume that $p_{\omega}(t) = \sum_{i \in \Zb} a_i t^i\neq 0$. Let $S:= \{i \in \Zb \mid a_i \neq 0\}$ and set $k:= \max{S} - \min{S}$.
Let $n \in \Nb$ and assume that $\xi := \exp(2\pi i l/n)$ for some $1 \leq l \leq n-1$ satisfies $p_{\omega}(\xi)=0$. Let $d$ be the degree of the minimal polynomial $m_{\xi}$ of $\xi$. If $\xi$ is a primitive $m$-th root of unity, then $m| n$ and
$d=\varphi(m)$, where $\varphi$ denotes Euler's $\varphi$-function. For some prime $p$ which divides $n$, we must have $(p-1) | \varphi(m)$ and hence ${\rm lpf}(n)-1 \leq d$.
Since $p_{\omega}(t)$ has rational coefficients, we also get that $m_{\xi} | p_{\omega}$ in the ring ${\mathbb Q}[t,t^{-1}]$ and hence $d \leq k$. Hence, if the assumption of Lemma \ref{large} fails then ${\rm lpf}(n) -1 \leq k$. This proves the claim.
\end{proof}

Let us discuss some examples to see how the previous results can be applied and what their limitations are.

\begin{example} \label{example1} The word map $\omega : SU(n)\times SU(n)\rightarrow SU(n)$  $\omega (a,b)=[a,b]^2$ is surjective for all $n \in \Nb$. Indeed, $p_{\omega}(t)=2(t-1)$ and ${\rm Ad}(p_{\omega}(\sigma))$ is a vector space automorphism of $\mathfrak{h}$ for $\sigma = (1,2,\dots,n)$.
\end{example}

\begin{example}\label{example2} The word map $\omega : SU(n)\times SU(n)\rightarrow SU(n)$ for $\omega (a,b)=a^2ba^{-1}ba^{-1}b^{-2}$ is surjective for all $n \in \Nb$. We have $p_{\omega}(t)=t^2+t-2$ and ${\rm Ad}(p_{\omega}(\sigma))$ is a vector space automorphism of $\mathfrak{h}$ for $\sigma = (1,2,\dots,n)$.
\end{example}

It is easy to see that $p_{\omega}$ vanishes for $\omega(a,b) = [a,b][a,b^{-1}]$ even though $\omega \not \in \Fb^{(2)}$. In this case we can still apply the method since we may interchange the role of $a$ and $b$ and note that $p_{\omega'} \neq 0$ for $\omega(a,b)=[b,a][b,a^{-1}]$. However, for $\omega(a,b) = [a,b][a,b^{-1}][a^{-1},b][a^{-1},b^{-1}]$ no such trick helps and we have to consider more complicated Nielsen transformations and their effect on our polynomial. We will show that for each $\omega \not \in \Fb^{(2)}$, there exists a basis for $\Fb$, such that with respect to the new basis, $p_{\omega} \neq 0$.  Any base change is induced by a sequence of Nielsen transformations. In Proposition \ref{ab} we study in detail how the base change $a \mapsto ab, b \mapsto b$ can be expressed in the natural basis of $\Fb^{(1)}$.

For $x,y \in G$, we use the notation $^{y}x := yxy^{-1}$. Note that this convention implies ${}^z({}^y x) = {}^{zy}x$ and ${}^z(xy) = {}^zx {}^zy$ as expected. It is well-known that for $x,y,z \in G$ we get:
\begin{equation} \label{deriv}
[x,yz]=[x,y]\cdot {}^y  [x,z] \quad \text{and} \quad [xy,z]={}^x [y,z]\cdot [x,z]. 
\end{equation}

From now on let us write $c:=ab$. Note that the set $\{c,b\}$ is a basis for $\Fb$. Our next goal is to express $[a^n,b^m]$ in terms of the commutators $[c^n,b^m]$, i.e. we want to determine the effect of the base change on the natural basis for $\Fb^{(1)}$. We will need the following lemma.

\begin{lem} Let $G$ be a group, $a,b \in G$ and $c:=ab$. Let $n,m \in \Zb$. Then, the following identities hold:
\begin{equation} \label{eq1}
 {}^a [c^n,b^{m}]=[c,b^{-1}][b^{-1},c^{n+1}][c^{n+1},b^{m-1}][b^{m-1},c].
\end{equation} 
and
\begin{equation} \label{eq1b}
 {}^{a^{-1}} [c^n,b^{m}]=[b,c^{n-1}] [c^{n-1},b^{m+1}] [b^{m+1},c^{-1}] [c^{-1},b].
\end{equation} 
\end{lem}
\begin{proof}
In order to prove \eqref{eq1}, we compute
\begin{eqnarray*}
{}^a [c^n,b^{m}] 
&=&cb^{-1}[c^n,b^{m}]bc^{-1} \\ 
&=& cb^{-1} c^nb^{m} c^{-n} b^{-m} b c^{-1} \\
&=& cb^{-1} c^{-1} b b^{-1} c^{n+1} b c^{-(n+1)} c^{n+1} b^{m-1} c^{-(n+1)} b^{-(m-1)} b^{m-1} c b^{-(m-1)} c^{-1} \\
&=& [c,b^{-1}][b^{-1},c^{n+1}][c^{n+1},b^{m-1}][b^{m-1},c].
\end{eqnarray*}
For \eqref{eq1b} we compute
\begin{eqnarray*} 
{}^{a^{-1}} [c^n,b^{m}]
&=& bc^{-1} c^n b^m c^{-n} b^{-m} cb^{-1}\\
&=& bc^{n-1}b^{-1} c^{-(n-1)} \cdot c^{n-1} b^{m+1} c^{-(n-1)} b^{-(m+1)} \cdot  b^{m+1}c^{-1} b^{-(m+1)} c \cdot  c^{-1} b c b^{-1} \\
&=& [b,c^{n-1}] [c^{n-1},b^{m+1}] [b^{m+1},c^{-1}] [c^{-1},b]
\end{eqnarray*}

This finishes the proof.
\end{proof}

Taking the inverse of Equation \eqref{eq1} we obtain for all $n,m \in \Zb$:
\begin{equation} \label{eq2}
 {}^a [b^m,c^n]= [c,b^{m-1}] [b^{m-1},c^{n+1}] [c^{n+1}, b^{-1}] [b^{-1},c].
\end{equation} 
and
\begin{equation} \label{eq2b}
{}^{a^{-1}} [b^m,c^n] = [b,c^{-1}][c^{-1},b^{m+1}][b^{m+1},c^{n-1}][c^{n-1},b].
\end{equation}

We are now ready to state and prove the technical heart of our computations.

\begin{proposition} \label{ab}
Let $G$ be a group, $a,b \in G$, $c:=ab$ and let $m  \in  \Zb$. If Then the following equations holds:
\begin{equation} \label{eq4}
 [a^n,b^m]= \prod_{i=1}^{n-1} [c^i,b^{-i}] [b^{-i}, c^{i+1}] \cdot \prod_{i=1}^{n} [c^{n+1-i},b^{m-n+i}] [b^{m-n+i}, c^{n-i}], \quad n \geq 1.
\end{equation}
\begin{equation} \label{eq5}
[a^{-n},b^m] = \prod_{i=1}^{n} [c^{1-i},b^i][b^i,c^{-i}] \cdot \prod_{i=1}^{n}  [c^{-(n+1)+i},b^{n+m+1-i}][b^{n+m+1-i},c^{-n+i}],\quad n \geq 1. 
\end{equation}
\end{proposition}

The main feature of the formulas above is the following. For $n \geq 1$, the powers of $b$ that appear expressing $[a^n,b^m]$ in the new basis will all be less or equal $\max\{-1,m\}$. At the same time, the powers of $a$ range between $1$ and $n$. The powers of $b$ that appear when expressing $[a^{-n},b^m]$ will be less or equal $\max\{n,n+m\}$ and if $m \geq 1$, then $[c^{-n},b^{n+m}][b^{n+m},c^{-n+1}]$ will appear exactly once. We will use this consequence in the proof of our main result.

\begin{proof}[Proof of Proposition \ref{ab}:]
We prove the claim \eqref{eq4} by induction on $n \in \Nb$. The claim is obviously true for $n=1$, since $[a,b^m]=[c,b^m]$.
Let $m \in \Nb$ and assume that the claim \eqref{eq4} is known for the pair $(n-1,m)$. We compute

\begin{eqnarray*}
&&[a^n,b^m] \\ &=& [a a^{n-1},b^m] \\
&\stackrel{\eqref{deriv}}{=}& {}^a[a^{n-1},b^m] [a, b^m]\\
&\stackrel{\eqref{eq4}}=&  \left(\prod_{i=1}^{n-2} {}^a[c^i,b^{-i}] {}^a[b^{-i} ,c^{i+1}] \cdot \prod_{i=1}^{n-1} {}^a[c^{n+1-i},b^{m-n+i}] {}^a[b^{m-n+i}, c^{n-i}] \right) [c, b^m]\\
&\stackrel{\eqref{eq1}+ \eqref{eq2}}{=}&
\prod_{i=1}^{n-2} [c,b^{-1}][b^{-1},c^{i+1}][c^{i+1},b^{-i-1}] [b^{-i-1},c^{i+2}] [c^{i+2}, b^{-1}] [b^{-1},c] \cdot \\
&& \prod_{i=1}^{n-1} [c,b^{-1}][b^{-1},c^{n-i+2}][c^{n-i+2},b^{m-n+i-1}] [b^{m-n+i-1},c^{n-i+1}] [c^{n-i+1}, b^{-1}] [b^{-1},c] \cdot \\
&& [c,b^m]\\
&=&\prod_{i=1}^{n-1} [c^i,b^{-i}] [b^{-i}, c^{i+1}] \cdot \prod_{i=1}^{n} [c^{n+1-i},b^{m-n+i}] [b^{m-n+i}, c^{n-i}].
\end{eqnarray*}

Now, we prove the claim \eqref{eq5} by induction on $n$. Again, the claim is true for $n=1$ since
$$
[a^{-1},b^m] =a^{-1}b^mab^{-m}  =  bc^{-1} b^m cb^{-1}b^{-m}=bc^{-1} b^{-1} c c^{-1} b^{m+1}  cb^{-(m+1)} =[b,c^{-1}][c^{-1},b^{m+1}].
$$
Let us assume that the claim \eqref{eq5} is known for the pair $(n-1,m)$. We compute:
\begin{eqnarray*}
&&[a^{-n},b^m]\\
&=&[a^{-1} a^{-(n-1)},b^m] \\
&\stackrel{\eqref{deriv}}{=}& {}^{a^{-1}}[a^{-(n-1)},b^m] \cdot [a^{-1},b^m]\\
&\stackrel{\eqref{eq5}}{=}& \prod_{i=1}^{n-1} {}^{a^{-1}}[c^{1-i},b^i]{}^{a^{-1}}[b^i,c^{-i}] \cdot \prod_{i=1}^{n-1}  {}^{a^{-1}}[c^{-(n+1)+i},b^{n+m+1-i}]{}^{a^{-1}}[b^{n+m+1-i},c^{-n+i}] \cdot \\
&& [b,c^{-1}][c^{-1},b^{m+1}]\\
&\stackrel{\eqref{eq1b} + \eqref{eq2b}}=&
\prod_{i=1}^{n-1} [b,c^{-i}] [c^{-i},b^{i+1}] [b^{i+1},c^{-1}][c^{-1},b][b,c^{-1}][c^{-1},b^{i+1}][b^{i+1},c^{-i-1}][c^{-i-1},b] \cdot \\
&&\prod_{i=1}^{n-1} [b,c^{-n+i-2}] [c^{-n+i-2},b^{n+m-i+2}][b^{n+m-i+2},c^{-n+i-1}][c^{-n+i-1},b] \cdot \\
&&[b,c^{-1}][c^{-1},b^{m+1}] \\
&=&
\prod_{i=1}^{n-1} [b,c^{-i}] [c^{-i},b^{i+1}][b^{i+1},c^{-i-1}][c^{-i-1},b] \cdot \\
&&\prod_{i=1}^{n-1} [b,c^{-n+i-2}] [c^{-n+i-2},b^{n+m-i+2}][b^{n+m-i+2},c^{-n+i-1}][c^{-n+i-1},b] \cdot \\
&&[b,c^{-1}][c^{-1},b^{m+1}] \\
&=& \prod_{i=1}^{n} [c^{1-i},b^i][b^i,c^{-i}] \cdot \prod_{i=1}^{n}  [c^{-(n+1)+i},b^{n+m+1-i}][b^{n+m+1-i},c^{-n+i}].
\end{eqnarray*}
This proves the claim.
\end{proof}

The following proposition shows that our previous computations are enough to deal with some more complicated words.

\begin{proposition}\label{id} Let $n \in \Nb$.
The word map $$\omega : SU(n)\times SU(n)\rightarrow SU(n)$$ for $\omega (a,b)=[a,b][a,b^{-1}][a^{-1},b][a^{-1},b^{-1}]$ is surjective. 
\end{proposition}
\begin{proof} 
Using Proposition $\ref{ab}$ we can write $\omega$ in the following form: $$\omega = [c,b][c,b^{-1}][b,c^{-1}][c^{-1},b^2][b,c^{-1}]$$ where $c=ab$. Indeed, $[a,b]=[c,b], [a,b^{-1}]=[c,b^{-1}], [a^{-1},b^{\pm 1}]=[b,c^{-1}][c^{-1},b^{1 \pm 1}]$ and hence
$\omega = [c,b][c,b^{-1}][b,c^{-1}][c^{-1},b^2][b,c^{-1}]$ as claimed. Now, we may compute $p_{\omega}$ with respect to the basis $\{b,c\}$ and obtain
$$p_{\omega}(t)= - (t-1)-(t^{-1}-1) - (t-1)+(t^2-1) -(t-1) = t^2 -3t-t^{-1}+3 \neq 0.$$
It is easy to see that if a root of unity $\xi$ satisfies $p_{\omega}(\xi)=0$, then $\xi=1$. Hence, Lemma \ref{large} implies that the word map associated with $\omega$ is surjective for all $n \in \Nb$.
\end{proof}

The key observation is that in the expression for $[a^i,b^j]$ in terms of $\{ [c^n,b^m] \mid n,m \in \Zb, nm \neq 0 \}$ can be used to isolate certain exponents. This will be used to show that for $\omega \not \in \Fb^{(2)}$, there is always \emph{some} basis such that $p_{\omega} \neq 0$.

\begin{proposition} \label{new}
Let $\Fb$ be the free group on two generators $\{a,b\}$ and $\omega \in \Fb^{(1)}$. If $\omega \not \in \Fb^{(2)}$, then there exists an basis of $\Fb$ such that $p_{\omega}(t) \neq 0,$ when computed with respect to this basis.
\end{proposition}

\begin{proof} 
The idea is to use the mechanism that is hidden in the proof of Proposition \ref{id}. Let $\omega \in \Fb^{(1)}$ and $\omega \notin \Fb^{(2)}$. Let us write
$$\omega =[a^{n_1},b^{m_1}]^{\nu_1}[a^{n_2},b^{m_2}]^{\nu_2} \cdots [a^{n_k},b^{m_k}]^{\nu_k}$$

Since $\omega \mapsto p_{\omega} \in \Zb[t,t^{-1}]$ is a homomorphism for \emph{any} basis of $\Fb$, we may freely rearrange the commutators in the product above. Moreover, we may assume that $(n_i , m_i)\neq (n_j,m_j)$ for $i \neq j$ and $\nu_i \in \Zb \setminus \{0\}$. 

Let $n:=\max \{|n_1|, \dots,|n_k|\}$. Without loss of generality we can assume that $n=-n_1$. Indeed, exchanging $a$ with $a^{-1}$, exchanges $n_i$ with $-n_i$, so that we may assume that $n_i<0$. Reordering the product allows to assume that $n=-n_1$. Again, reordering does not change $p_{\omega}$ in any basis, since only the class of $\omega$ in $\Fb^{(1)}/\Fb^{(2)}$ matters in our computation. In addition, we may assume that there exists $k' \in \Nb$ such that $n=n_1 = n_2 = \cdots = n_{k'}$ and $n \neq n_l$ for $l >k'$. Without loss generality, we have $m_1 > m_2 > \cdots> m_{k'}$ and set $m:=m_1$. Upon possibly replacing $b$ by $b^{-1}$, we may assume that $m>0$.

Let us now set $c:=ab \in \Fb$. We will now analyze how $\omega$ is written in terms of the basis $\{c,b\}$. By Proposition \ref{ab}, each factor $[a^{-n},b^{m_i}]$ of $\omega$ contains factors $[c^{-n},b^{n}]^{-1}$ and $[c^{-n},b^{n+m_i}]$ and these are the only factors in $\omega$ of the form $[c^{-n'}, b^k]$ for some $k \in \Zb$ and $n' \geq n$. Repeating this process, we can set $c_q := ab^q$ and consider the basis $\{c_q,b\}$. With respect to this basis $\omega$ will contain a factor of the form $[c_q^{-n}, b^{qn+m}]$. From Proposition \ref{ab} and the remarks after its statement, we conclude that for $q \in \Nb$ high enough, the factor $[c_q^{-n}, b^{qn+m}][b^{qn+m},c_q^{-n+1}]$ will be the only appearance of $b^{qn+m}$. Hence, computing $p_{\omega}$ with respect to the basis $\{b,c_q\}$, the coefficient of $t^{qn+m}$ will be non-zero. This proves the claim.
 \end{proof}

We are now ready to prove Theorem \ref{mainresult}.

\begin{proof}[Proof of Theorem \ref{mainresult}:]
Let $\omega \in \Fb \setminus \Fb^{(2)}$ be arbitrary.
If $\omega \not \in \Fb^{(1)}$, then Lemma \ref{notzero} proves the claim. Hence, we may assume $\omega \in \Fb^{(1)}$ and $\omega \not \in \Fb^{(2)}$. By Proposition \ref{new}, there exists a basis of $\Fb$ such that $p_{\omega}(t)\neq 0$. The claim follows from Corollary \ref{corlem}.
\end{proof}
\section{Special families of words and open problems}
\label{engelwords}

\subsection{Engel words}
In this last section we study Engel words and show that the associated word maps are always surjective. Corresponding results for finite simple groups were proved in \cite{Engel}.

\begin{definition} Let $\Fb$ be the free group on two generators $\{a,b\}$.
The $k$-th Engel word $e_k(a,b) \in \Fb$ is defined recursively by the equations:
\begin{eqnarray*} e_0(a,b)&=&a,
\\ e_k(a,b)&=& [e_{k-1},b], \quad k\geq1.
\end{eqnarray*}
\end{definition}

For a group $G$, the corresponding map $e_k:G\times G \rightarrow G$ is called the $k$-th Engel word map. We are now ready to prove Theorem \ref{engeltheorem}
\begin{proof}[Proof of Theorem \ref{engeltheorem}:]
We want to compute $p_{e_k}$ with respect to the basis $\{b,a\}$.
First of all, it is easy to see that $^{b}[b^{m},a^{n}]=[b^{m+1},a^{n}][b,a^{n}]^{-1}$ for $n,m \in \mathbb{Z}$. Indeed, we just compute
\begin{eqnarray*}
^{b}[b^{m},a^{n}] &=& bb^ma^nb^{-m}a^{-n}b^{-1} \\
&=& b^{m+1} a^n b^{-(m+1)} a^{-n} a^n b a^{-n} b^{-1}\\
&=&[b^{m+1},a^{n}][b,a^{n}]^{-1}.
\end{eqnarray*}
This shows that if $p_{\omega}(t) = \sum_{i} a_i t^i$, then
$p_{{}^b \omega}(t) = \sum_{i} a_i t (t^{i} -1) = t p_{\omega}(t) - tp_{\omega}(1)$. Hence, 
$$p_{[\omega,b]}(t) = p_{\omega} (t) - p_{{}^b \omega}(t) = (1-t) p_{\omega}(t) + t p_{\omega}(1).$$
Since $p_{e_1}(t) = p_{[b,a]^{-1}}(t) = 1-t$ we conclude that $p_{e_k}(t) = (1-t)^k$ for all $k \in \Nb$. Lemma \ref{large} implies the claim.
\end{proof}

\subsection{Open problems}
It is clear that the method presented in this paper has serious limitations and cannot possible work for words $\omega \in \Fb^{(2)}$. A first non-trivial case is $\omega(a,b)=[[a,b],[a^2,b^2]]$. It is unknown to us if Larsen's question has a positive answer for this word.

\begin{question} Let $\Fb$ be the free group on two generators $\{a,b\}$ and let $\omega = [[a,b],[a^2,b^2]]$. Is the associated word map $\omega \colon SU(n) \times SU(n) \to SU(n)$ surjective for all but finitely many $n \in \Nb$.
\end{question}

If $\omega \not\in \Fb^{(2)}$ it would be desirable to find out if the restrictions on $n \in \Nb$ in Theorem \ref{mainresult} are necessary. We are not aware of a word $\omega \not \in \Fb^{(2)}$, where the associated word map is not surjective for \emph{all} $n \in \Nb$. In order to understand this problem, we need to understand the map $\omega \mapsto p_{\omega}$ more directly. We can endow $\Zb[t,t^{-1}]$ with a $\Zb \Fb$-module structure such that $a \cdot f(t) = tf(t)$ and $b \cdot f(t) = f(t)$ for all $f(t) \in \Zb[t,t^{-1}]$. The quotient $\Fb^{(1)}/\Fb^{(2)}$ is also a $\Zb \Fb$-module, where the module structure is induced from the conjugation action.
Since $p_{\omega}$ is well-defined on $\Fb^{(1)}/\Fb^{(2)}$, it is natural to study the induced map
$$\bar p \colon \Fb^{(1)}/\Fb^{(2)} \to \Zb[t,t^{-1}].$$
\begin{lem}
The map $\bar p$ is a homomorphism of $\Zb \Fb$-modules.
\end{lem}
\begin{proof} We denote the class of $[a^n,b^m]$ in $\Fb^{(1)}/\Fb^{(2)}$ by $\xi_{n,m}$. Hence, 
$\bar p(\xi_{n,m}) = m(t^n-1)$ by definition.
It follows from the equations $^{a}[a^{n},b^{m}]=[a^{n+1},b^{m}][a,b^{m}]^{-1}$ and
$^{b}[a^n,b^m]=[a^{n},b]^{-1} [a^{n}, b^{m+1}]$
that
$\bar p(a \cdot \xi_{n,m})= \bar p(\xi_{n+1,m} - \xi_{1,m}) = m(t^{n+1}-t) = t \cdot \bar p(\xi_{n,m})$ and
$\bar p(b \cdot \xi_{n,m})=\bar p(\xi_{n,m+1} - \xi_{n,1}) = m(t^n-1) = \bar p(\xi_{n,m})$. This finishes the proof.
\end{proof}

For $\xi \in \Fb^{(1)}/\Fb^{(2)}$ and a general automorphism $\alpha \in {\rm Aut}(\Fb)$, the relation between $\bar p(\xi)$ and $\bar p(\alpha(\xi))$ remains obscure.

\begin{question} Let $\xi \in \Fb^{(1)}/ \Fb^{(2)}$. Is there an automorphism $\alpha \in {\rm Aut}(\Fb)$ such that the only root of the polynomial $\bar p(\alpha(\xi))$ which is a root of unity is equal to one.
\end{question}

A positive answer to this question would remove the restrictions on $n \in \Nb$ in Theorem \ref{mainresult}.

\section*{Acknowledgment}

The research of A.T.\ was supported by ERC. 
A.E. wants to thank the IMPRS Leipzig and the MPI-MIS Leipzig for support and an excellent research environment.


\end{document}